\documentclass[10pt]{article}

\usepackage{amsmath,amssymb}

\usepackage{diagbox}

\usepackage{mathrsfs}
\usepackage[amsmath, thmmarks]{ntheorem}

\usepackage{hyperref}

\hypersetup{colorlinks=true,allcolors=[rgb]{0,0.4,1},citecolor=red}

\usepackage{tikz}
\usepackage{graphicx}
\usetikzlibrary{calc,decorations.markings}

\usepackage{array}
\usepackage{makecell}

\newtheorem{theorem}{Theorem}[section]

\newtheorem{corollary}[theorem]{Corollary}

\newtheorem{claim}[theorem]{Claim}

\theorembodyfont{\normalfont}
\newtheorem{definition}[theorem]{Definition}
\newtheorem*{remark}{Remark}
\newtheorem{example}[theorem]{Example}

\newtheorem*{result}{Result}

\theoremheaderfont{\itshape}
\theoremseparator{.}
\theoremsymbol{\ensuremath{\blacksquare}}
\newtheorem*{proof}{Proof}

\newcommand{\BC}{\mathbb{C}}

\newcommand{\fg}{\mathfrak{g}}

\newcommand{\fh}{\mathfrak{h}}

\newcommand{\fn}{\mathfrak{n}}

\newcommand{\SL}{\mathfrak{sl}}
\newcommand{\GL}{\mathfrak{gl}}

\title{New approaches to $\GL_N$ weight system}
\author{Zhuoke Yang\thanks{Higher School of Economics,
partially supported by International Laboratory of Cluster Geometry NRU HSE, RF Government grant,
ag. № 075-15-2021-608 dated 08.06.2021}}
\date{}

\begin{document}
\maketitle

\begin{abstract}
The present paper has been motivated by an aspiration for understanding the weight system corresponding to the Lie algebra $\GL_N$. The straightforward approach to computing the values of a Lie algebra weight system on a general chord diagram amounts to elaborating calculations in the noncommutative universal enveloping algebra, in spite of the fact that the result belongs to the center of the latter. The first approach is based on a suggestion due to M. Kazarian to define an invariant of permutations taking values in the center of the universal enveloping algebra of $\GL_N$. The restriction of this invariant to involutions without fixed points (such an involution determines a chord diagram) coincides with the value of the $\GL_N$ -weight system on this chord diagram. We describe the recursion allowing one to compute the $\GL_N$ -invariant of permutations and demonstrate how it works in a number of examples. The second approach is based on the Harish-Chandra isomorphism for the Lie algebras $\GL_N$. This isomorphism identifies the center of the universal enveloping algebra $\GL_N$ with the ring $\Lambda^*(N)$ of shifted symmetric polynomials in N variables. The Harish-Chandra projection can be applied separately for each monomial in the defining polynomial of the weight system; as a result, the main body of computations can be done in a commutative algebra, rather than noncommutative one.
\end{abstract}
\newpage
\tableofcontents

\section{List of Symbols}

\begin{tabular}{cp{0.8\textwidth}}
  $\fg, \langle\cdot,\cdot\rangle$ & a Lie algebra endowed with a nondegenerate invariant bilinear product\\
  $\GL_N$ & general linear Lie algebra; consists of all $N\times N$ matrices with the commutator serving as the Lie bracket \\
  $\SL_N$ & special linear Lie algebra; consists of all $N\times N$ trace-free matrices with the commutator serving as the Lie bracket \\
  $d$ & dimension of Lie algebra; specifically, for $\GL_N$, $d=N^2$ \\
  $D$ & a chord diagram \\
  $n$ & the number of chords in a chord diagram  \\
  $K_n$ & the chord diagram with $n$ chords any two of which intersect one another\\
  $\pi$ & the projection to the subspace of primitive elements in the Hopf algebra of chord diagrams
  whose kernel is the subspace of decomposable elements\\
  $C_1,\cdots,C_N$ & Casimir elements in $U(\GL_N)$\\
  $w$ & a weight system\\
  $w_\fg$ & the Lie algebra weight system associated to a Lie algebra $\fg$\\
  $\bar{w}_\fg$ & $w_\fg(\pi(\cdot))$; the composition of the Lie algebra weight system $w_\fg$
  with the projection~$\pi$ to the subspace of primitives\\
  $\sigma$ & a permutation\\
  $m$ & the number of permutated elements; e.g. for the permutation determined by a chord diagram, $m=2n$\\
  $G(\sigma)$ & the digraph of the permutation~$\sigma$\\
  $\Lambda^*(N)$ & the algebra of shifted symmetric polynomials in~$N$ variables\\
  $\phi$ & the Harish–Chandra projection\\
  $p_1,\cdots,p_N$ & shifted power sum polynomials\\
\end{tabular}\\

\newpage
\section{Introduction}

In V.~A.~Vassiliev's theory of finite type knot invariants,
a weight system can be associated to each such invariant.
A weight system is a function on chord diagrams satisfying so-called
$4$-term relations.

In the opposite direction, according to a Kontsevich theorem,
to each weight system taking values in a field of characteristic~$0$,
a finite type knot invariant can be associated in a canonical way.
This makes studying weight systems an important part of knot theory.

There is a number of approaches to constructing weight systems. In particular,
a huge class of weight systems can be constructed from metrized finite dimensional Lie algebras.
The present paper has been motivated by an aspiration for understanding the weight
system corresponding to the Lie algebra~$\GL_N$.

The straightforward approach to computing the values of a Lie algebra weight system
on a general chord diagram amounts to elaborating calculations in the noncommutative
universal enveloping algebra, in spite of the fact that the result belongs to the
center of the latter. This approach is rather inefficient even for the simplest
noncommutative Lie algebra $\SL_2$, whose weight system is associated to the
knot invariant known as the colored Jones polynomial. For this Lie algebra, however, there is
a recurrence relation due to S.~Chmutov and A.~Varchenko~\cite{ChV},
and numerous computations have been done using it, see e.g.~\cite{F1,F2,Za}.
In particular, recently, values of the $\SL_2$-weight system have been
computed on certain nontrivial infinite families of chord diagrams.

Much less is known about other Lie algebras; for them, explicit answers have been computed
only for chord diagrams of very small order or for simple families
of chord diagrams, see~\cite{ZY}. In particular, no recurrence similar
to the Chmutov--Varchenko one exists (with the exception of the Lie superalgebra $\GL_{1|1}$,
see~\cite{FKV,ChL}). The goal of the present paper is to provide
two new ways to compute the values of the $\GL_N$ weight system.

The first approach is based on a suggestion due to M.~Kazarian
to define an invariant of permutations taking values in the center
of the universal enveloping algebra of $\GL_N$. The restriction of this invariant
to involutions without fixed points (such an involution determines a chord
diagram) coincides with the value of the $\GL_N$-weight system on this chord diagram.
We describe the recursion allowing one to compute the $\GL_N$-invariant
of permutations and demonstrate how it works in a number of examples.

For~$N'<N$, the center of the universal enveloping algebra of $\GL_{N'}$
is naturally embedded into that of $\GL_N$, and the $\GL_N$-weight system
is stable: its value on a permutation is a universal polynomial.
The recursion we describe allows one to compute this polynomial
simultaneously for all~$N$.

The calculations of the highers homogeneous part of the universal $\GL_N$ weight system in this terms of Casimir elements for some special primitive elements given by open Jacobi diagrams were the central part of the lower estimate for the dimension of the Vassiliev knot invariants in \cite{CD,Dasbach} (see also \cite[\S 14.5.4]{Chmutov2012}).

The second approach is based on the Harish-Chandra isomorphism for the Lie
algebras~$\GL_N$. This isomorphism identifies the center of the universal enveloping
algebra $\GL_N$ with the ring $\Lambda^*(N)$ of shifted symmetric polynomials
in~$N$ variables. The Harish-Chandra projection can be applied separately for
each monomial in the defining polynomial of the weight system;
as a result, the main body of computations can be done in a commutative
algebra, rather than noncommutative one.

The paper is organized as follows.
In Sec.~\ref{sec:def}, we recall the construction of Lie algebra weight systems.
In Sec.~\ref{sec:rec}, we
describe an extension of the $\GL_N$-weight system to arbitrary permutations
and a recursion to computing its values on permutations.
In Sec.~\ref{sec:sym}, we apply,
for the Lie algebras $\GL_N$, the Harish-Chandra
isomorphism  to develop one more algorithm for computing the corresponding weight system.
We compare the results with those obtained by the previous method.
In Sec.~\ref{sec:ha}, we recall the Hopf algebra structure on the
space of chord diagrams modulo $4$-term relations, and
discuss the behaviour of the $\GL_N$-weight system with respect to this structure.

The author is grateful to M.~Kazarian and G.~Olshanskii for valuable suggestions,
and to S.~Lando for permanent attention.

%\newpage
\section{Definition of $\GL_N$ weight system}\label{sec:def}

Below, we use standard notions from the theory of finite order knot invariants;
see, e.g.~\cite{Chmutov2012}.

A {\it chord diagram\/} of order~$n$ is an oriented circle (called the \emph{Wilson loop}) endowed with~$2n$
pairwise distinct points split into~$n$ disjoint pairs, considered up to
orientation-preserving diffeomorphisms of the circle.

%The {\it intersection graph\/} $\gamma(D)$ of a chord diagram~$D$ is the simple graph whose
%vertices are in one-to-one correspondence with the chords in~$D$,
%and two vertices are connected by an edge iff the corresponding chords intersect
%one another.

A~{\it weight system\/} is a function $w$ on chord diagrams satisfying the $4$-term
relation; see Fig.~\ref{fourtermrelation}.

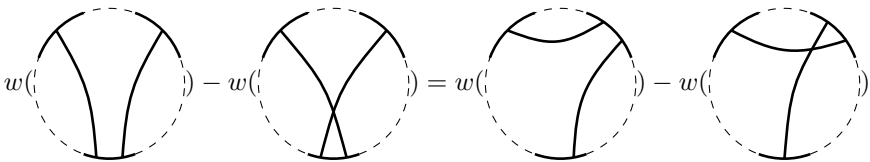
\begin{figure}[ht]
\[
w(\begin{tikzpicture}[baseline={([yshift=-.5ex]current bounding box.center)}]
	\draw[dashed] (0,0) circle (1);
	\draw[line width=1pt]  ([shift=( 20:1cm)]0,0) arc [start angle= 20, end angle= 70, radius=1];
	\draw[line width=1pt]  ([shift=(110:1cm)]0,0) arc [start angle=110, end angle=160, radius=1];
	\draw[line width=1pt]  ([shift=(250:1cm)]0,0) arc [start angle=250, end angle=290, radius=1];
	\draw[line width=1pt] (45:1) ..  controls (5:0.3) and (-40:0.3)  .. (280:1);
	\draw[line width=1pt] (135:1) ..  controls (175:0.3) and (220:0.3)  .. (260:1);
\end{tikzpicture})  -
w(\begin{tikzpicture}[baseline={([yshift=-.5ex]current bounding box.center)}]
	\draw[dashed] (0,0) circle (1);
	\draw[line width=1pt]  ([shift=( 20:1cm)]0,0) arc [start angle= 20, end angle= 70, radius=1];
	\draw[line width=1pt]  ([shift=(110:1cm)]0,0) arc [start angle=110, end angle=160, radius=1];
	\draw[line width=1pt]  ([shift=(250:1cm)]0,0) arc [start angle=250, end angle=290, radius=1];
	\draw[line width=1pt] (45:1) ..  controls (-5:0.1) and (-50:0.1)  .. (260:1);
	\draw[line width=1pt] (135:1) ..  controls (185:0.1) and (225:0.1)  .. (280:1);
\end{tikzpicture})  =
w(\begin{tikzpicture}[baseline={([yshift=-.5ex]current bounding box.center)}]
	\draw[dashed] (0,0) circle (1);
	\draw[line width=1pt]  ([shift=( 20:1cm)]0,0) arc [start angle= 20, end angle= 70, radius=1];
	\draw[line width=1pt]  ([shift=(110:1cm)]0,0) arc [start angle=110, end angle=160, radius=1];
	\draw[line width=1pt]  ([shift=(250:1cm)]0,0) arc [start angle=250, end angle=290, radius=1];
	\draw[line width=1pt] (35:1) ..  controls (0:0.3) and (-45:0.3)  .. (280:1);
	\draw[line width=1pt] (135:1) ..  controls (105:0.5) and (85:0.5)  .. (55:1);
\end{tikzpicture})  -
w(\begin{tikzpicture}[baseline={([yshift=-.5ex]current bounding box.center)}]
	\draw[dashed] (0,0) circle (1);
	\draw[line width=1pt]  ([shift=( 20:1cm)]0,0) arc [start angle= 20, end angle= 70, radius=1];
	\draw[line width=1pt]  ([shift=(110:1cm)]0,0) arc [start angle=110, end angle=160, radius=1];
	\draw[line width=1pt]  ([shift=(250:1cm)]0,0) arc [start angle=250, end angle=290, radius=1];
	\draw[line width=1pt] (55:1) ..  controls (5:0.1) and (-40:0.1)  .. (270:1);
	\draw[line width=1pt] (135:1) ..  controls (105:0.4) and (65:0.4)  .. (35:1);
\end{tikzpicture})
\]
\caption{$4$-term relation}
\label{fourtermrelation}
\end{figure}

In figures, the outer circle of the chord diagram is always assumed to be
oriented counterclockwise. Dashed arcs may contain ends of arbitrary sets
of chords, same for all the four terms in the picture.

\begin{definition}
The {\it product} of two chord diagrams $D_1$ and $D_2$ is defined by cutting and gluing the two circles as shown

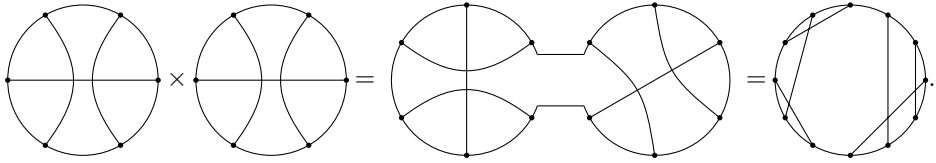
\begin{tikzpicture}[baseline={([yshift=-.5ex]current bounding box.center)}]
	\draw (0,0) circle (1);
	\draw (1,0) -- (-1,0);
	\fill[black] (1,0) circle (1pt)
                 (-1,0) circle (1pt)
                 (60:1) circle (1pt)
                 (120:1) circle (1pt)
                 (240:1) circle (1pt)
                 (300:1) circle (1pt);
	\draw (60:1)  ..  controls (0,0.2) and (0,-0.2)  .. (300:1);
	\draw (120:1) ..  controls (0,0.2) and (0,-0.2)  .. (240:1);
\end{tikzpicture} $\times$
\begin{tikzpicture}[baseline={([yshift=-.5ex]current bounding box.center)}]
	\draw (0,0) circle (1);
	\draw (1,0) -- (-1,0);
	\fill[black] (1,0) circle (1pt)
                 (-1,0) circle (1pt)
                 (60:1) circle (1pt)
                 (120:1) circle (1pt)
                 (240:1) circle (1pt)
                 (300:1) circle (1pt);
	\draw (60:1)  ..  controls (0,0.2) and (0,-0.2)  .. (300:1);
	\draw (120:1) ..  controls (0,0.2) and (0,-0.2)  .. (240:1);
\end{tikzpicture} $=$
\begin{tikzpicture}[baseline={([yshift=-.5ex]current bounding box.center)}]
	\draw ([shift=( 20:1cm)]0,0) arc [start angle= 20, end angle= 340, radius=1];;
	\draw (0,1) -- (0,-1);
	\fill[black] (0,1) circle (1pt)
                 (0,-1) circle (1pt)
                 (30:1) circle (1pt)
                 (150:1) circle (1pt)
                 (210:1) circle (1pt)
                 (330:1) circle (1pt);
	\draw (30:1)  ..  controls (0.2,0) and (-0.2,0)  .. (150:1);
	\draw (210:1) ..  controls (-0.2,0) and (0.2,0)  .. (330:1);

	\draw[xshift=2.5cm] ([shift=( 200:1cm)]0,0) arc [start angle= 200, end angle=520, radius=1];
	\draw[xshift=2.5cm] (30:1) -- (210:1);
	\fill[xshift=2.5cm,black] (0,1) circle (1pt)
                 (0,-1) circle (1pt)
                 (30:1) circle (1pt)
                 (150:1) circle (1pt)
                 (210:1) circle (1pt)
                 (330:1) circle (1pt);
	\draw[xshift=2.5cm] (150:1)  ..  controls (190:0.2) and (230:0.2)  .. (0,-1);
	\draw[xshift=2.5cm] (0,1) ..  controls (50:0.2) and (10:0.2)  .. (330:1);
	\draw (20:1) -- ($(20:1)+(0.62,0)$);
	\draw (-20:1) -- ($(-20:1)+(0.62,0)$);
\end{tikzpicture} $=$
\begin{tikzpicture}[baseline={([yshift=-.5ex]current bounding box.center)}]
	\draw (0,0) circle (1);
	\draw (  0:1) -- (-90:1)
          ( 30:1) -- (-30:1)
          ( 60:1) -- (-60:1)
          ( 90:1) -- (150:1)
          (120:1) -- (210:1)
          (180:1) -- (240:1);
	\fill[black] (  0:1) circle (1pt)
                 ( 30:1) circle (1pt)
                 ( 60:1) circle (1pt)
                 ( 90:1) circle (1pt)
                 (120:1) circle (1pt)
                 (150:1) circle (1pt)
                 (180:1) circle (1pt)
                 (210:1) circle (1pt)
                 (240:1) circle (1pt)
                 (270:1) circle (1pt)
                 (300:1) circle (1pt)
                 (330:1) circle (1pt);
\end{tikzpicture}.

Modulo $4$-term relations, the product is well-defined, see\cite[\S 4.4.3]{Chmutov2012}.
\end{definition}

Given a Lie algebra $\fg$ equipped with a non-degenerate invariant bilinear form, one can construct a weight system with values in the center of its universal enveloping algebra $U(\fg)$.
This is the form M. Kontsevich~\cite{kon1993} gave to a construction due to D. Bar-Natan~\cite{bar1995vassiliev}.
Kontsevich’s construction proceeds as follows.

\begin{definition}[Universal Lie algebra weight system] Let $\fg$ be a metrized Lie algebra over $\mathbb{R}$ or $\mathbb{C}$, that is, a Lie algebra with an ad-invariant non-degenerate bilinear form $\langle\cdot,\cdot\rangle$.
	Let~$d$ denote the dimension of~$\fg$. Choose a basis $e_1 ,\dots,e_d$ of $\fg$ and let $e_1^* ,\dots,e_d^*$ be the dual basis with respect to the form $\langle\cdot,\cdot\rangle$, $\langle e_i,e_j^*\rangle=\delta_{ij}$, where $\delta$ is the Kronecker delta.

Given a chord diagram $D$ with $n$ chords, we first choose a base point on the circle, away from the ends of the chords of $D$. This gives a linear order on the endpoints of the chords, increasing in the positive direction of the Wilson loop. Assign to each chord $a$ an index, that is, an integer-valued variable, $i_a$. The values of $i_a$ will range from $1$ to $d$, the dimension of the Lie algebra. Mark the first endpoint of the chord~$a$ with the symbol $e_{i_a}$ and the second endpoint with $e_{i_a}^{*}$.

Now, write the product of all the $e_{i_a}$ and all the $e_{i_a}^{*}$ , in the order in which they appear on the Wilson loop of $D$, and take the sum of the $d^n$ elements of the universal enveloping algebra $U(\fg)$ obtained by substituting all possible values of the indices $i_a$ into this product. Denote by $w_\fg (D)$ the resulting element of $U(\fg)$.
\end{definition}

\begin{claim}~{\rm\cite{kon1993}} The function $w_\fg :D\mapsto w_\fg (D)$ on chord diagrams has the following properties:
\begin{enumerate}
\item the element $w_\fg (D)$ does not depend on the choice of the base point on the diagram;
\item it does not depend on the choice of the basis ${e_i }$ of the Lie algebra~$\fg$;
\item its image belongs to the ad-invariant subspace
\[
	U(\fg)^{\fg}=\{x\in U(\fg)|xy=yx \text{ for all } y\in \fg\}=ZU(\fg);
\]
\item it is multiplicative, $w_\fg (D_1D_2)=w_\fg (D_1)w_\fg (D_2)$ for any pair of chord diagrams $D_1,D_2$;
\item this map from chord diagrams to $ZU(\fg)$ satisfies the 4-term relations.
\end{enumerate}
\end{claim}

Consider the Lie algebra $\GL_N $ of all $N \times N$ matrices. Fix the trace of the product of matrices as the preferred ad-invariant form: $\langle x,y\rangle = \text{Tr}(xy)$.
The Lie algebra $\GL_N $ is linearly spanned by matrix units $E_{ij}$ having $1$ on the intersection of $i$~th row with $j~$th column and~$0$ elsewhere, $i,j=1,\dots,N$. We have $\langle E_{ij},E_{kl}\rangle = \delta_{il}\delta_{jk}$. Therefore, the duality between $\GL_N $ and $\GL_N^*$ defined by $\langle \cdot,\cdot\rangle$ is given by the formula $E_{ij}^*= E_{ji}$.
The commutation relations for $\GL_N$ have the form
\begin{eqnarray}\label{e1}
[E_{kl},E_{ji}]=E_{kl}E_{ji}-E_{ji}E_{kl}=\delta_{lj}E_{ki}-\delta_{ik}E_{jl}.
\end{eqnarray}

Now, the straightforward computation of the value of the $\GL_N$ weight system
looks like follows.

\begin{example}
For a chord diagram $K_1$ with a single chord, we have
$$
w_{\GL_N}(K_1)=\sum\limits_{i,j=1}^{N}E_{ij}E_{ji}.
$$
We denote this element by~$C_2\in ZU(\GL_N)$ and call the \emph{second Casimir}.
Similarly, $\sum\limits_{i=1}^{N}E_{ii}=C_1$, and, more generally,
$$
C_k=\sum\limits_{i_1,i_2,\dots, i_k=1}^{N}E_{i_1i_2}E_{i_2i_3}\dots E_{i_ki_1}
$$
is the $k$~th\emph{ Casimir element} in $ZU(\GL_N)$.
\end{example}

The center $ZU(\GL_N)$ is isomorphic to the ring
of polynomials in the Casimir elements $C_1,\dots,C_N$: $ZU(\GL_N)=\BC[C_1,\dots,C_N]$,
see \cite{Zhe}. The higher Casimir elements $C_{N+1},C_{N+2},\dots$ can be represented as polynomials 
in~$C_1,\dots,C_N$. The value $w_{\GL_N}(D)$
of the $\GL_N$-weight system on a chord diagram~$D$ with~$n$ chords is a polynomial
in $C_1,\dots,C_n$.

\begin{example}
For the chord diagram, which we denote by~$K_2$, since its intersection graph is $K_2$,
the complete graph on~$2$ vertices, we have

\begin{tikzpicture}[baseline={([yshift=-.5ex]current bounding box.center)}]
	\draw (0,0) circle (1);
	\draw ( 45:1) -- (225:1)
          (135:1) -- (315:1);
	\fill[black] ( 45:1) circle (1pt) node[right] {$E_{ij}$}
                 (135:1) circle (1pt) node[left] {$E_{kl}$}
                 (225:1) circle (1pt) node[left] {$E_{ji}$}
                 (315:1) circle (1pt) node[right] {$E_{lk}$};
\end{tikzpicture}
\ \ \ \ \ $w_{\GL_N}(K_2)=\sum\limits_{i,j,k,l=1}^{N}E_{ij}E_{kl}E_{ji}E_{lk}.$ \\
Using the commutation relations~(\ref{e1}) we obtain
\begin{align*}
	w_{\GL_N}(K_2)&=\sum_{i,j,k,l=1}^{N}E_{ij}E_{kl}E_{ji}E_{lk} \\
	&=\sum_{i,j,k,l=1}^{N}E_{ij}E_{ji}E_{kl}E_{lk}+\sum_{i,j,k,l=1}^{N}\delta_{lj}E_{ij}E_{ki}E_{lk}-\sum_{i,j,k,l=1}^{N}\delta_{ik}E_{ij}E_{jl}E_{lk}\\
	&=C_2^2+\sum_{i,j,k=1}^{N}E_{ij}E_{ki}E_{jk}-\sum_{i,j,l=1}^{N}E_{ij}E_{jl}E_{li}\\
	&=C_2^2+\sum_{i,j,k=1}^{N}E_{ij}[E_{ki},E_{jk}]\\
	&=C_2^2+\sum_{i,j,k=1}^{N}\delta_{ij}E_{ij}E_{kk}-\sum_{i,j,k=1}^{N}\delta_{kk}E_{ij}E_{ji}\\
	&=C_2^2+C_1^2-NC_2.
\end{align*}
\end{example}

Even in this simple example, the straightforward computation includes a lot of steps.
A much more efficient algorithm is suggested in the next section.

%\remark $w_{\GL_N}(K_2)=C_2^2+C_1^2-NC_2$ and $w_{\GL_N}(\text{isolated chord})=C_2$, there is no leaf lemma of $\GL_N$ weight system. the reason is $\GL_N$ is not semisimple Lie algebra.

%\newpage
\section{The $\GL$ weight system for permutations}
\label{sec:rec}

There is no recurrence relation for the weight system $w_{\GL_N}$ we know about. Instead,
following the suggestion by M.~Kazarian, we interpret an arc diagram
as an involution without fixed points on the set of its ends and
extend the function  $w_{\GL_N}$ to arbitrary permutations of any number of permutated elements. For permutations,
in contrast to chord diagrams, such a recurrence relation could be given.

For a permutation $\sigma\in S_m$, set

\[
	w_{\GL_N}(\sigma)=\sum_{i_1,\cdots,i_m=1}^N E_{i_1i_{\sigma(1)}}E_{i_2i_{\sigma(2)}}\cdots E_{i_mi_{\sigma(m)}}\in U(\GL_N).
\]
We claim that
\begin{itemize}
	\item $w_{\GL_N}$ lies in the center of $U(\GL_N)$.
	\item this element is invariant under conjugation by  a cyclic permutation:

		$w_{\GL_N}(\sigma)=\sum_{i_1,\cdots,i_m=1}^N E_{i_2i_{\sigma(2)}}\cdots E_{i_mi_{\sigma(m)}}E_{i_1i_{\sigma(1)}}$.
\end{itemize}
	
For example, the standard generator 
$$C_m=\sum^N_{i_1,\cdots,i_m=1}E_{i_1i_2}E_{i_2i_3}\cdots E_{i_{m-1}i_m}E_{i_mi_1}$$
corresponds to the cyclic permutation $1\mapsto2\mapsto\cdots\mapsto m\mapsto1\in S_m$.

%This value on the fixed point-free involution in $S_n$, $n=2m$, defined by a chord diagram
%with~$m$ chords, coincides with the value of $w_{\GL_N}$ on this chord diagram.
%Similarly to the case of chord diagrams, the value $w_{\GL_N}(\sigma)$, for~$N$ sufficiently large,
%is independent of~$N$; we denote this common value by~$w_{\GL}(\sigma)$.

On the other hand, a chord diagram with $n$ chords can be considered as an involution without fixed points on a set of $m=2n$ elements. The value of $w_{\GL_N}$ on the corresponding involution is equal to the value of the $\GL_N$ weight system on the corresponding chord diagram.
\begin{example}
For the chord diagram $K_n=$
\begin{tikzpicture}[baseline={([yshift=-.5ex]current bounding box.center)}]
			\draw[->,thick] (-2,0)--(2,0);
		\draw (-1.8,0) ..  controls (-1,.5) ..(.2,0);
		\draw (-1.4,0) ..  controls (-.6,.5) ..(.6,0);
		\draw (-.2,0) ..  controls (1,.5) ..(1.8,0);
		\fill[black] (-1.8,0) circle (1pt) node[below] {\tiny 1};
		\fill[black] (-1.4,0) circle (1pt) node[below] {\tiny 2};
		\fill[black] (-0.2,0) circle (1pt) node[below] {\tiny n};
		\fill[black] ( .15,0) circle (1pt) node[below] {\tiny n+1};
		\fill[black] ( .6,0) circle (1pt)  node[below]{\tiny n+2};
		\fill[black] ( 1.8,0) circle (1pt) node[below] {\tiny 2n};
		\node[below] at (-.8,0) {$\cdots$};
		\node[below] at (1.2,0) {$\cdots$};
\end{tikzpicture}
we have
\begin{align*}
	w_{\GL_N}(K_n)&=\sum_{i_1,\cdots,i_{2n}=1}^N E_{i_1i_{n+1}}E_{i_2i_{n+2}}\cdots E_{i_ni_{2n}}E_{i_{n+1}i_1}E_{i_{n+2}i_2}\cdots E_{i_{2n}i_n}\\
  &=w_{\GL_N}(\left(1\ n+1)(2\ n+2)\cdots(n\ 2n)\right)
\end{align*}
\end{example}

\begin{definition}[digraph of the permutation]
	Let us represent  a permutation as an oriented graph.
The $m$ vertices of the graph correspond to the permuted elements.
They are ordered cyclically and are placed on a real line,
subsequently connected with horizontal arrows looking right and numbered
from left to right. The arc arrows
show the action of the permutation (so that each vertex is incident with exactly
one incoming and one outgoing arc edge).
	The digraph $G(\sigma)$ of a permutation $\sigma\in S_m$
consists of these~$m$ vertices and $m$ oriented edges, for example:
	\[
	G(\left(1\ n+1)(2\ n+2)\cdots(n\ 2n)\right)=\begin{tikzpicture}[baseline={([yshift=-.5ex]current bounding box.center)},decoration={markings, mark= at position .55 with {\arrow{stealth}}}]
		\draw[->,thick] (-2,0)--(2,0);
		\draw[blue,postaction={decorate}] (-1.8,0) ..  controls (-1,.5) ..(.2,0);
		\draw[blue,postaction={decorate}] (-1.4,0) ..  controls (-.6,.5) ..(.6,0);
		\draw[blue,postaction={decorate}] (-.2,0) ..  controls (1,.5) ..(1.8,0);
		\draw[blue,postaction={decorate}] (.2,0) ..  controls (-1,-.5) ..(-1.8,0);
		\draw[blue,postaction={decorate}] (.6,0) ..  controls (-.6,-.5) ..(-1.4,0);
		\draw[blue,postaction={decorate}] (1.8,0) ..  controls (1,-.5) ..(-.2,0);
		\fill[black] (-1.8,0) circle (1pt) node[below] {\tiny 1};
		\fill[black] (-1.4,0) circle (1pt) node[below] {\tiny 2};
		\fill[black] (-0.2,0) circle (1pt) node[below] {\tiny n};
		\fill[black] ( .15,0) circle (1pt) node[below] {\tiny n+1};
		\fill[black] ( .6,0) circle (1pt)  node[below]{\tiny n+2};
		\fill[black] ( 1.8,0) circle (1pt) node[below] {\tiny 2n};
		\node[below] at (-.8,0) {$\cdots$};
		\node[below] at (1.2,0) {$\cdots$};
\end{tikzpicture}
	\]
\end{definition}

\begin{example}
  The digraph of the Casimir element $C_m$, which corresponds to the cyclic permutation $1\mapsto2\mapsto\cdots\mapsto m\mapsto1\in S_m$,
   is the following one:
    \[
  G(\left(1\ 2\ 3\cdots m-1\ m)\right)=\begin{tikzpicture}[baseline={([yshift=-.5ex]current bounding box.center)},decoration={markings, mark= at position .55 with {\arrow{stealth}}}]
    \draw[->,thick] (-2,0)--(2,0);
    \draw[blue,postaction={decorate}] (-1.8,0) ..  controls (-1.5,.5) ..(-1.2,0);
    \draw[blue,postaction={decorate}] (-1.2,0) ..  controls (-.9,.5) ..(-.6,0);
    \draw[blue,postaction={decorate}] (1.2,0) ..  controls (1.5,.5) ..(1.8,0);
    \draw[blue,postaction={decorate}] (.6,0) ..  controls (.9,.5) ..(1.2,0);
    \draw[blue,postaction={decorate}] (1.8,0) ..  controls (0,-.5) ..(-1.8,0);
    \fill[black] (-1.8,0) circle (1pt) node[below] {\tiny 1};
    \fill[black] (-1.2,0) circle (1pt) node[below] {\tiny 2};
    \fill[black] (-.6,0) circle (1pt) node[below] {\tiny 3};
    \fill[black] (1.8,0) circle (1pt) node[below] {\tiny m};
    \fill[black] (1.2,0) circle (1pt) node[below] {\tiny m-1};
    \fill[black] (.6,0) circle (1pt) node[below] {\tiny m-2};
    \node[above] at (0,0) {$\cdots$};
\end{tikzpicture}
  \]
\end{example}

\begin{theorem}
The value of the $w_{\GL_N}$ invariant of permutations possesses the following properties:
\begin{itemize}
	\item for an empty graph (with no vertices) the value of $w_{\GL_N}$ is equal to $1$,
		$w_{\GL_N}(\textcircled{})=1$;
	\item $w_{\GL_N}$ is multiplicative with respect to concatenation of permutations;
	\item for a cyclic permutation {\rm(}with the cyclic order on the set of permuted elements  compatible with the permutation{\rm)},
 the value of $w_{\GL_N}$ is the standard generator,
		$w_{\GL_N}(1\mapsto2\mapsto\cdots\mapsto k\mapsto1)=C_k$.

\item {\rm(}\textbf{Recurrence Rule}{\rm)} For the graph of an arbitrary permutation $\sigma$ in~$S_m$,
and for any two neighboring elements $k,k+1$, of the permuted set $\{1,2,\dots,m\}$, we have
for the value of the $w_{\GL_N}$ weight system
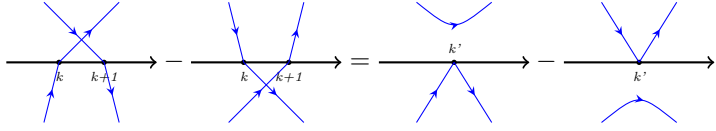
\begin{equation*}
	\begin{tikzpicture}[baseline={([yshift=-.5ex]current bounding box.center)},decoration={markings, mark= at position .55 with {\arrow{stealth}}}]
		\draw[->,thick] (-1,0) --  (1,0);
		\fill[black] (-.3,0) circle (1pt) node[below] {\tiny k};
		\fill[black] (.3,0) circle (1pt) node[below] {\tiny k+1};
		%\draw (-.5,.8) node[left] {a};
		%\draw (-.5,-.8) node[left] {b};
		%\draw (.5,.8) node[right] {c};
		%\draw (.5,-.8) node[right] {d};
		\draw[blue,postaction={decorate}] (-.5,.8) -- (.3,0);
		\draw[blue,postaction={decorate}] (-.3,0) -- (.5,.8);
		\draw[blue,postaction={decorate}] (-.5,-.8) -- (-.3,0);
		\draw[blue,postaction={decorate}] (.3,0) -- (.5,-.8);
		%\draw (0,-1.2) node { $\sigma$};
	\end{tikzpicture}-
	\begin{tikzpicture}[baseline={([yshift=-.5ex]current bounding box.center)},decoration={markings, mark= at position .55 with {\arrow{stealth}}}]
		\draw[->,thick] (-1,0) --  (1,0);
		\fill[black] (.3,0) circle (1pt) node[below] {\tiny k+1};
		\fill[black] (-.3,0) circle (1pt) node[below] {\tiny k};
		%\draw (-.5,.8) node[left] {a};
		%\draw (-.5,-.8) node[left] {b};
		%\draw (.5,.8) node[right] {c};
		%\draw (.5,-.8) node[right] {d};
		\draw[blue,postaction={decorate}] (-.5,.8) -- (-.3,0);
		\draw[blue,postaction={decorate}] (.3,0) -- (.5,.8);
		\draw[blue,postaction={decorate}] (-.5,-.8) -- (.3,0);
		\draw[blue,postaction={decorate}] (-.3,0) -- (.5,-.8);
		%\draw (0,-1.2) node { $\sigma^{*}$};
	\end{tikzpicture}=
	\begin{tikzpicture}[baseline={([yshift=-.5ex]current bounding box.center)},decoration={markings, mark= at position .55 with {\arrow{stealth}}}]
		\draw[->,thick] (-1,0)  -- (1,0);
		\fill[black] (0,0) circle (1pt) node[above] {\tiny k'};
		%\draw (-.5,.8) node[left] {a};
		%\draw (-.5,-.8) node[left] {b};
		%\draw (.5,.8) node[right] {c};
		%\draw (.5,-.8) node[right] {d};
		\draw[blue,postaction={decorate}] (-.5,.8) ..controls (0,.4) .. (.5,.8);
		\draw[blue,postaction={decorate}] (-.5,-.8) -- (0,0);
		\draw[blue,postaction={decorate}] (0,0) -- (.5,-.8);
		%\draw (0,-1.2) node { $\sigma'$};
	\end{tikzpicture}-
	\begin{tikzpicture}[baseline={([yshift=-.5ex]current bounding box.center)},decoration={markings, mark= at position .55 with {\arrow{stealth}}}]
		\draw[->,thick] (-1,0) --  (1,0);
		\fill[black] (0,0) circle (1pt) node[below] {\tiny k'};
		%\draw (-.5,.8) node[left] {a};
		%\draw (-.5,-.8) node[left] {b};
		%\draw (.5,.8) node[right] {c};
		%\draw (.5,-.8) node[right] {d};
		\draw[blue,postaction={decorate}] (-.5,-.8) ..controls (0,-.4) .. (.5,-.8);
		\draw[blue,postaction={decorate}] (-.5,.8) -- (0,0);
		\draw[blue,postaction={decorate}] (0,0) -- (.5,.8);
		%\draw (0,-1.2) node { $\sigma''$};
	\end{tikzpicture}
\end{equation*}

In the diagrams on the left, two horizontally neighboring vertices
and the edges incident to them are depicted, while
on the right these two vertices are replaced with a single one;
the other vertices are placed somewhere on the line and their positions are the same on all
diagrams participating in the relations, but the numbers of the vertices to the right
of the latter are to be decreased by~$1$.

In particular, for the special case $\sigma(k+1)=k$, the recurrence looks like follows:
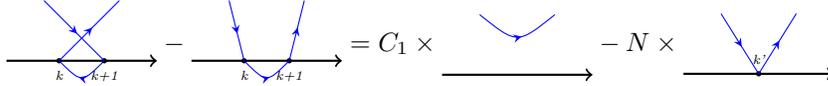
\begin{equation*}
	\begin{tikzpicture}[baseline={([yshift=-.5ex]current bounding box.center)},decoration={markings, mark= at position .55 with {\arrow{stealth}}}]
		\draw[->,thick] (-1,0) --  (1,0);
		\fill[black] (-.3,0) circle (1pt) node[below] {\tiny k};
		\fill[black] (.3,0) circle (1pt) node[below] {\tiny k+1};
		%\draw (-.5,.8) node[left] {a};
		%\draw (.5,.8) node[right] {b};
		\draw[blue,postaction={decorate}] (-.5,.8) -- (.3,0);
		\draw[blue,postaction={decorate}] (-.3,0) -- (.5,.8);
		\draw[blue,postaction={decorate}] (.3,0) ..controls(0,-.3).. (-.3,0);
		%\draw (0,-0.7) node { $\sigma$};
	\end{tikzpicture}-
	\begin{tikzpicture}[baseline={([yshift=-.5ex]current bounding box.center)},decoration={markings, mark= at position .55 with {\arrow{stealth}}}]
		\draw[->,thick] (-1,0) --  (1,0);
		\fill[black] (.3,0) circle (1pt) node[below] {\tiny k+1};
		\fill[black] (-.3,0) circle (1pt) node[below] {\tiny k};
		%\draw (-.5,.8) node[left] {a};
		%\draw (.5,.8) node[right] {b};
		\draw[blue,postaction={decorate}] (-.5,.8) -- (-.3,0);
		\draw[blue,postaction={decorate}] (.3,0) -- (.5,.8);
		\draw[blue,postaction={decorate}] (-.3,0) ..controls(0,-.3).. (.3,0);
		%\draw (0,-0.7) node { $\sigma*$};
	\end{tikzpicture}=C_1\times
	\begin{tikzpicture}[baseline={([yshift=-.5ex]current bounding box.center)},decoration={markings, mark= at position .55 with {\arrow{stealth}}}]
		\draw[->,thick] (-1,0)  -- (1,0);
		%\draw (-.5,.8) node[left] {a};
		%\draw (.5,.8) node[right] {b};
		\draw[blue,postaction={decorate}] (-.5,.8) ..controls (0,.4) .. (.5,.8);
		%\draw (0,-0.4) node { $\sigma'$};
	\end{tikzpicture}-N\times
	\begin{tikzpicture}[baseline={([yshift=-.5ex]current bounding box.center)},decoration={markings, mark= at position .55 with {\arrow{stealth}}}]
		\draw[->,thick] (-1,0) --  (1,0);
		\fill[black] (0,0) circle (1pt) node[above] {\tiny k'};
		%\draw (-.5,.8) node[left] {a};
		%\draw (.5,.8) node[right] {b};
		\draw[blue,postaction={decorate}] (-.5,.8) -- (0,0);
		\draw[blue,postaction={decorate}] (0,0) -- (.5,.8);
		%\draw (0,-0.4) node { $\sigma''$};
	\end{tikzpicture}
\end{equation*}
These relations are indeed a recursion, that is, they allow one to
replace the computation of $w_{\GL_N}$ on a permutation with its computation on simpler permutations.
\end{itemize}

\end{theorem}

\begin{proof}
We only need to prove the Recurrence Rule, which is just the graphical explanation of the Lie bracket in $\GL_N$.
\begin{eqnarray*} E_{i_ki_{\sigma(k)}}E_{i_{k+1}i_{\sigma({k+1})}}-E_{i_{k+1}i_{\sigma(k+1)}}E_{i_ki_{\sigma(k)}}
&=&[E_{i_ki_{\sigma(k)}},E_{i_{k+1}i_{\sigma(k+1)}}]\\
&=&\delta_{i_{\sigma(k)}i_{k+1}}E_{i_ki_{\sigma(k+1)}}-\delta_{i_{\sigma(k+1)}i_k}E_{i_{k+1}i_{\sigma(k)}}.
\end{eqnarray*}

In the special case, when $\sigma(k+1)=k$, we have
\[
	E_{i_ki_{\sigma(k)}}E_{i_{k+1}i_k}-E_{i_{k+1}i_k}E_{i_ki_{\sigma(k)}}=[E_{i_ki_{\sigma(k)}},E_{i_{k+1}i_k}]
=\delta_{i_{\sigma(k)}i_{k+1}}E_{i_ki_k}-\delta_{i_ki_k}E_{i_{k+1}i_{\sigma(k)}}.
\]
When summing it from $i_1,\cdots,i_m=1$ to $N$, we obtain $\sum \delta_{i_{\sigma(k)}i_l}E_{i_ki_k}=C_1\sum\delta_{i_{\sigma(k)}i_l}$ and $\sum\delta_{i_ki_k}E_{i_li_{\sigma(k)}}=N\sum E_{i_li_{\sigma(k)}}$.

The second graph on the left hand side corresponds to a permutation obtained from the first one
 by a conjugation with a transposition of two neighbouring vertices. Both graphs on the right hand side have smaller number of vertices.
 Applying these relations, every graph can be reduced to a monomial in the variables $C_k$ (a concatenation of cyclic permutations)
 modulo terms of smaller degrees. This provides an inductive computation of the invariant $w_{\GL_N}$.
\end{proof}

\begin{remark}
In the situation of permutations corresponding to chord diagrams, the difference at the right-hand side of the recurrence relation represents a Jacobi diagram with a triple vertex according to the STU relation from \cite{bar1995vassiliev,Chmutov2012}. This gives a way to calculate the weight system $w_{\GL_N}$ on primitive elements given by Jacobi diagrams. For some special elements the calculation of this sort were given in \cite{CD, Dasbach}.
\end{remark}

\begin{corollary}
The value of $w_{\GL_N}$ on a permutation is well defined, can be represented as a polynomial
in $N,C_1,C_2,\dots$, and this polynomial is universal.
\end{corollary}

\begin{definition}[universal $\GL$-weight system on permutations]
  The {\it universal $\GL$-weight system on permutation} $w_\GL$ is the weight system
  taking values in the polynomial ring $\BC[N,C_1,C_2,\cdots]$, which satisfies  $w_\GL(\sigma)=w_{\GL_N}(\sigma)$, 
  for all permutations $\sigma$ and is obtained by the above recurrence relations.
\end{definition}

\begin{example}
Let us compute the value of $w_{\GL}$ on the cyclic permutation $(1\ 3\ 2)$ by switching the places of node $2$ and $3$:
\begin{equation*}
  \begin{tikzpicture}[baseline={([yshift=-.5ex]current bounding box.center)},decoration={markings, mark= at position .55 with {\arrow{stealth}}}]
    \draw[->,thick] (-1,0) --  (1,0);
    \fill[black] (-.5,0) circle (1pt) node[below] {\tiny 1};
    \fill[black] (.5,0) circle (1pt) node[below] {\tiny 3};
    \fill[black] (0,0) circle (1pt) node[below] {\tiny 2};
    \draw[blue,postaction={decorate}] (-.5,0) ..controls(0,.3).. (.5,0);
    \draw[blue,postaction={decorate}] (.5,0) ..controls(.25,-.3).. (0,0);
    \draw[blue,postaction={decorate}] (0,0) ..controls(-.25,-.3).. (-.5,0);
    \draw (0,-0.7) node { $(1\ 3\ 2)$};
    \draw (0,0.8) node {};
  \end{tikzpicture}-
  \begin{tikzpicture}[baseline={([yshift=-.5ex]current bounding box.center)},decoration={markings, mark= at position .55 with {\arrow{stealth}}}]
    \draw[->,thick] (-1,0) --  (1,0);
    \fill[black] (-.5,0) circle (1pt) ;
    \fill[black] (.5,0) circle (1pt) ;
    \fill[black] (0,0) circle (1pt) ;
    \draw[blue,postaction={decorate}] (.5,0) ..controls(0,-.3).. (-.5,0);
    \draw[blue,postaction={decorate}] (-.5,0) ..controls(-.25,.3).. (0,0);
    \draw[blue,postaction={decorate}] (0,0) ..controls(.25,.3).. (.5,0);
    \draw (0,-0.7) node { $(1\ 2\ 3)$};
    \draw (0,0.8) node {};
  \end{tikzpicture}=C_1\times
  \begin{tikzpicture}[baseline={([yshift=-.5ex]current bounding box.center)},decoration={markings, mark= at position .55 with {\arrow{stealth}}}]
    \fill[black] (0,0) circle (1pt) ;
    \draw[->,thick] (-1,0)  -- (1,0);
    \draw[blue,postaction={decorate}] (0,0) ..controls (-.5,.7) and (.5,.7)   .. (0,0);
    \draw (0,-0.47) node { $(1)$};
  \end{tikzpicture}-N\times
  \begin{tikzpicture}[baseline={([yshift=-.5ex]current bounding box.center)},decoration={markings, mark= at position .55 with {\arrow{stealth}}}]
    \draw[->,thick] (-1,0) --  (1,0);
    \fill[black] (-.4,0) circle (1pt) node[above] {};
    \fill[black] (.4,0) circle (1pt) node[above] {};
    \draw[blue,postaction={decorate}] (-.4,0) ..controls(0,-.3).. (.4,0);
    \draw[blue,postaction={decorate}] (.4,0) ..controls(0,.3).. (-.4,0);
    \draw (0,0.8) node {};
    \draw (0,-0.7) node { $(1\ 2)$};
  \end{tikzpicture}
\end{equation*}
\begin{eqnarray*}
w_\GL((1\ 3\ 2))&=&w_\GL((1\ 2\ 3))+C_1\times w_\GL((1))-N\times w_\GL((1\ 2))\\
                &=&C_3+C_1^2-NC_2
\end{eqnarray*}
\end{example}

The reader will find below a table of values of the $\GL$ weight system on chord diagrams $K_n$,
which have $n$ chords and each chord crosses each other; these results were obtained by computer calculation. 
These diagrams are chosen because computation
of Lie algebra weight systems on them is extremely nontrivial, even for the Lie algebra $\SL_2$,
where we know the Chmutov–Varchenko recurrence relation.
In addition, these diagrams generate a Hopf subalgebra of the Hopf algebra
of chord diagrams, see Sec.~\ref{sec:ha},
which allows us to compute the $\GL$ weight system on the projection of $K_n$ to the primitive space.

\begin{result}
	\begin{align*}
		w_{\GL}(K_2)&=-NC_2+C_1^2+C_2^2\\
		w_{\GL}(K_3)&=2 C_2 N^2+(-2 C_1^2-3 C_2^2) N+C_2^3+3 C_1^2 C_2\\
		w_{\GL}(K_4)&=-6 C_2 N^3+(6 C_1^2+11 C_2^2-2 C_3) N^2+(-6 C_2^3-14 C_1^2 C_2+6 C_1 C_2-2 C_2+2 C_4) N\\
					  &\phantom{=}+3 C_1^4-4 C_1^3+6 C_2^2 C_1^2+2 C_1^2-8 C_3 C_1+C_2^4+6 C_2^2\\
		w_{\GL}(K_5)&=24 C_2 N^4+(-24 C_1^2-50 C_2^2+24 C_3) N^3\\
					  &\phantom{=}+(35 C_2^3+70 C_1^2 C_2-72 C_1 C_2-10 C_3 C_2+32 C_2-24 C_4) N^2\\
					  &\phantom{=}+(-20 C_1^4+48 C_1^3-50 C_2^2 C_1^2-32 C_1^2+30 C_2^2 C_1+96 C_3 C_1-10 C_2^4-82 C_2^2+10 C_2 C_4) N\\
					  &\phantom{=}+C_2^5+10 C_1^2 C_2^3+30 C_2^3+15 C_1^4 C_2-20 C_1^3 C_2+10 C_1^2 C_2-40 C_1 C_3 C_2\\
		w_{\GL}(K_6)&=-120 C_2 N^5+(120 C_1^2+274 C_2^2-240 C_3) N^4\\
					  &\phantom{=}+(-225 C_2^3-404 C_1^2 C_2+720 C_1 C_2+174 C_3 C_2-416 C_2+224 C_4) N^3\\
					  &\phantom{=}+(130 C_1^4-480 C_1^3+375 C_2^2 C_1^2-30 C_3 C_1^2+416 C_1^2-522 C_2^2 C_1\\
					  &\phantom{=}-896 C_3 C_1+85 C_2^4+1014 C_2^2-30 C_2^2 C_3-88 C_3-174 C_2 C_4+32 C_5) N^2\\
					  &\phantom{=}+(-15 C_2^5-130 C_1^2 C_2^3+90 C_1 C_2^3-552 C_2^3+30 C_4 C_2^2-165 C_1^4 C_2+438 C_1^3 C_2-492 C_1^2 C_2\\
					  &\phantom{=}+264 C_1 C_2+696 C_1 C_3 C_2+64 C_3 C_2-72 C_2+30 C_1^2 C_4-160 C_1 C_4+88 C_4-16 C_6) N\\
					  &\phantom{=}+15 C_1^6-60 C_1^5+45 C_2^2 C_1^4+150 C_1^4-60 C_2^2 C_1^3-120 C_3 C_1^3-176 C_1^3+15 C_2^4 C_1^2\\
					  &\phantom{=}+120 C_2^2 C_1^2+256 C_3 C_1^2+72 C_1^2-192 C_2^2 C_1-120 C_2^2 C_3 C_1-352 C_3 C_1\\
					  &\phantom{=}+96 C_5 C_1+C_2^6+90 C_2^4+264 C_2^2+160 C_3^2-240 C_2 C_4\\
		w_{\GL}(K_7)&=720 C_2 N^6+(-720 C_1^2-1764 C_2^2+2400 C_3) N^5\\
					  &\phantom{=}+(1624 C_2^3+2688 C_1^2 C_2-7200 C_1 C_2-2324 C_3 C_2+5264 C_2-1856 C_4) N^4\\
					  &\phantom{=}+(-924 C_1^4+4800 C_1^3-2954 C_2^2 C_1^2+644 C_3 C_1^2-5264 C_1^2\\
					  &\phantom{=}+6972 C_2^2 C_1+7424 C_3 C_1-735 C_2^4-12892 C_2^2+714 C_2^2 C_3+3392 C_3+2212 C_2 C_4-1088 C_5) N^3\\
					  &\phantom{=}+(175 C_2^5+1365 C_1^2 C_2^3-2142 C_1 C_2^3-70 C_3 C_2^3+8358 C_2^3-714 C_4 C_2^2+1540 C_1^4 C_2\\
					  &\phantom{=}-6580 C_1^3 C_2+11736 C_1^2 C_2-10176 C_1 C_2-210 C_1^2 C_3 C_2-8848 C_1 C_3 C_2-2792 C_3 C_2+224 C_5 C_2\\
					  &\phantom{=}+3456 C_2-644 C_1^2 C_4+5440 C_1 C_4-3392 C_4+544 C_6) N^2\\
					  &\phantom{=}+(-210 C_1^6+1288 C_1^5-735 C_2^2 C_1^4-4412 C_1^4+2058 C_2^2 C_1^3+2576 C_3 C_1^3+6784 C_1^3-280 C_2^4 C_1^2\\
					  &\phantom{=}-4704 C_2^2 C_1^2-8704 C_3 C_1^2+210 C_2 C_4 C_1^2-3456 C_1^2+210 C_2^4 C_1+8376 C_2^2 C_1+2856 C_2^2 C_3 C_1\\
					  &\phantom{=}+13568 C_3 C_1-1120 C_2 C_4 C_1-3264 C_5 C_1-21 C_2^6-2212 C_2^4-10680 C_2^2-4096 C_3^2+448 C_2^2 C_3\\
					  &\phantom{=}+70 C_2^3 C_4+7432 C_2 C_4-112 C_2 C_6)N\\
					  &\phantom{=}+504 C_1^2 C_2-1232 C_1^3 C_2+1050 C_1^4 C_2-420 C_1^5 C_2+105 C_1^6 C_2+3192 C_2^3-1344 C_1 C_2^3\\
					  &\phantom{=}+700 C_1^2 C_2^3-140 C_1^3 C_2^3+105 C_1^4 C_2^3+210 C_2^5+21 C_1^2 C_2^5+C_2^7-5152 C_1 C_2 C_3+1792 C_1^2 C_2 C_3\\
					  &\phantom{=}-840 C_1^3 C_2 C_3-280 C_1 C_2^3 C_3+1120 C_2 C_3^2+1344 C_1^2 C_4-1680 C_2^2 C_4+672 C_1 C_2 C_5		 		
	\end{align*}
\end{result}

\begin{remark}
The Lie algebra $\GL_N$ is not simple. Instead, it is a direct sum of a commutative one-dimensional
Lie algebra and a simple Lie algebra $\SL_N$. The one-dimensional commutative Lie subalgebra in $\GL_N$
consists of scalar matrices, which are $\BC$-multiples of the identity matrix. Therefore, the center
$ZU(\GL_N)$ of the universal enveloping algebra of $\GL_N$ is the tensor product of the center of the
universal enveloping algebra of~$\BC$ and that of $\SL_N$, whence the ring of polynomials in the first Casimir~$C_1$
with coefficients in $ZU(\SL_N)$. Therefore, the values of the weight system $w_{\SL_N}$ can be computed
from that of $w_{\GL_N}$ by setting $C_1=0$. In the result, $C_2,C_3,\dots$ denote the projections of
the corresponding Casimir elements in $ZU(\GL_N)$ to $ZU(\SL_N)$.
\end{remark}

%\newpage
\section{Symmetric functions and Harish–Chandra isomorphism}
\label{sec:sym}

In this section, we make use of the  Harish-Chandra isomorphism
for the Lie algebras $\GL_N$ to compute the corresponding weight systems.

\begin{definition}[algebra of shifted symmetric polynomials]\ \\
For a positive integer $N$, the algebra $\Lambda^*(N)$ of shifted symmetric
polynomials in~$N$ variables $x_1,x_2,\cdots,x_N$ consists of polynomials
that are invariant under changes of variables
$$
(x_1,\cdots,x_i,x_{i+1},\cdots,x_N)\mapsto(x_1,\cdots,x_{i+1}-1,x_{i}+1,\cdots,x_N),
$$
for all $i=1,\cdots,N-1$.
Equivalently, this is the algebra of symmetric polynomials in the shifted variables $(x_1-1,x_2-2,\dots,x_N-N)$.
\end{definition}

The universal enveloping algebra $U(\GL_N)$ of the Lie algebra $\GL_N$ admits
the direct sum decomposition
\begin{equation}\label{HCd}
	U(\GL_N)=(\fn_-U(\GL_N)+U(\GL_N)\fn_+)\oplus U(\fh),
\end{equation}
where $\fn_-$ and $\fn_+$ are
 the nilpotent subalgebras of, respectively, upper and lower triangular matrices in $\GL_N$,
 and $\fh$ is the subalgebra of diagonal matrices.

\begin{definition}[Harish–Chandra projection in $U(\GL_N)$]\ \\
The Harish–Chandra projection for $U(\GL_N)$ is the projection to the second summand in~(\ref{HCd})
\[
	\phi:U(\GL_N)\to U(\fh)=\BC[E_{11},\cdots,E_{NN}],
\]
where $E_{11},\cdots,E_{NN}$ are the diagonal matrix units in $\GL_N$; they commute with one another.
\end{definition}
\begin{theorem}[Harish–Chandra isomorphism~\cite{Zhe,Ols1996}]
	The Harish–Chandra projection, when restricted to the center $ZU(\GL_N)$,
is an algebra isomorphism to the algebra $\Lambda^*(N)\subset U(\fh)$ of shifted symmetric polynomials in $E_{11},\cdots,E_{NN}$.
\end{theorem}

Thus, the computation of the value of the $\GL_N$ weight system on a chord diagram
can be elaborated by applying the Harish-Chandra projection to each monomial of
the polynomial. For such a monomial, the projection can be computed by moving
variables $E_{ij}$ with $i>j$ to the left, and/or variables $E_{ij}$ with $i<j$ to the right
by means of applying the commutator relations. If, in the process, we obtain
monomials in $\fn_-U(\GL_N)$ or $U(\GL_N)\fn_+$, then we replace such a monomial
with $0$. A monomial in the (mutually commuting) variables $E_{ii}$ cannot be simplified, and its projection to
$U(\fh)$ coincides with itself. The resulting polynomial in $E_{11},\cdots,E_{NN}$ will be automatically shifted symmetric.

%For the specific relationships between the Casimir elements $C_1,\cdots,C_N$ and power sums $p_1,\cdots,p_N$.

\begin{example}
	Let's compute the projection of the quadratic Casimir element
	\[
		C_2=\sum_{i,j} E_{ij}E_{ji}\in ZU(\GL_N)
	\]
to $U(\fh)$. We have
\begin{eqnarray*}
		C_2&=&\sum_i E_{ii}^2+\sum_{i< j}E_{ij}E_{ji}+\sum_{i> j}E_{ij}E_{ji}\\
&=&
\sum_i E_{ii}^2+2\sum_{i>j} E_{ij}E_{ji}+\sum_{i<j}[E_{ij},E_{ji}]\\
&=&
\sum_i E_{ii}^2+2\sum_{i>j} E_{ij}E_{ji}+\sum_{i<j}(E_{ii}-E_{jj}).
\end{eqnarray*}
In this expression, the first and the third summand	depend on the diagonal unit
elements $E_{ii}$ only, while the second summand is in $\fn_-U(\GL_N)+U(\GL_N)\fn_+$, whence
the image under the projection is
	\begin{align*}
\phi(C_2)=
		&\sum_i E_{ii}^2+\sum_{i<j}(E_{ii}-E_{jj})\\
		=&\sum_{i}(E_{ii}^2+(N+1-2i)E_{ii}).
		\end{align*}
\end{example}

Similarly to the ring of ordinary symmetric functions, the ring $\Lambda^*(N)$ of shifted symmetric functions in~$N$ variables
is isomorphic to a polynomial ring in~$N$ variables.
There is a variety of convenient $N$-tuples of generators in $\Lambda^*(N)$.
One of them is the tuple of shifted power sum polynomials
$$
p_k=\sum_{i}\left(\left(E_{ii}+\frac{N+1}2-i\right)^k-\left(\frac{N+1}2-i\right)^k\right).
$$

Representing $\phi(C_2)$ in the form
$$
\phi(C_2)=\sum_{i}\left(\left(E_{ii}+\frac{N+1}{2}-i\right)^2-\left(\frac{N+1}{2}-i\right)^2\right),
$$
we see that it is just~$p_2$.

\begin{remark}
Since the Harish–Chandra isomorphism can be applied to arbitrary elements of $ZU(\GL_N)$, we can also apply it
to the values of $w_\GL$ on permutations.
\end{remark}

For $k>2$,  the expression for $\phi(C_k)$ is not reduced to just linear combinations of power sums.
If fact, we have the following explicit formula, which follows from 
\cite{PerelomovPopov1968}, \cite[\S~60]{Zhe} and \cite[Remark 2.1.20]{Ols1991},
\begin{eqnarray*}
1-Nu-\sum_{k=1}^\infty
\phi(C_{k})u^{k+1}&=&\prod_{i=1}^N\frac{1-(E_{ii}+N-i+1)u}{1-(E_{ii}+N-i)u}\\
&=&(1-Nu)e^{\sum_{k=1}^\infty\frac{(1-\frac{N-1}{2}u)^{-k}-(1-\frac{N+1}{2}u)^{-k}}{k}u^kp_k }.
\end{eqnarray*}
This provides an expression for the image $\phi(C_k)$ of $C_k$ as a polynomial in $p_1,p_2,\dots$, which is valid for all~$N$. The projections of the Casimir elements $C_1,\cdots,C_N$ to $U(\fh)$
can be expressed in shifted power sums $p_1,\cdots,p_N$ in the following way:
	\begin{align*}
	 	\phi(C_1)&=p_1\\
	 	\phi(C_2)&=p_2\\
	 	\phi(C_3)&=-\frac{1}{4} N^2 p_1+\frac{N p_2}{2}+\frac{p_1}{4}+p_3-\frac{p_1^2}{2}\\
\phi(C_4)&=-\frac{1}{4} N^3 p_1+N \left(-\frac{p_1^2}{2}+\frac{p_1}{4}+p_3\right)-p_1 p_2+\frac{p_2}{2}+p_4\\	 	\cdots
	\end{align*}
Computations using the Harish-Chandra isomorphism also are elaborative,
and the results they produce are not universal, they depend on~$N$. 
It is more efficient, therefore, to substitute the known values $\phi(C_k)$
into the answers obtained by the previous method.

For the values of the $\GL$ weight system on the chord diagrams~$K_n$, this yields

\begin{result}
  \begin{align*}
    {w}_{\GL}(K_2)&=-N p_2+p_1^2+p_2^2\\
    {w}_{\GL}(K_3)&=2 N^2 p_2+N \left(-2 p_1^2-3 p_2^2\right)+p_2^3+3 p_1^2 p_2\\
    {w}_{\GL}(K_4)&=-7 N^3 p_2+N^2 \left(8 p_1^2+11 p_2^2\right)+N \left(-6 p_2^3-14 p_1^2 p_2-p_2+2 p_4\right)\\
                  &\phantom{=}+3 p_1^4+6 p_2^2 p_1^2-8 p_3 p_1+p_2^4+6 p_2^2\\
    {w}_{\GL}(K_5)&=36 N^4 p_2+N^3 \left(-48 p_1^2-55 p_2^2\right)+N^2 \left(35 p_2^3+80 p_1^2 p_2+20 p_2\right)\\
                  &\phantom{=}+N \left(-20 p_1^4-50 p_2^2 p_1^2-8 p_1^2-10 p_2^4-77 p_2^2+10 p_2 p_4-24 p_4\right)\\
                  &\phantom{=}+p_2^5+10 p_1^2 p_2^3+30 p_2^3+15 p_1^4 p_2-40 p_1 p_3 p_2+96 p_1 p_3\\
    {w}_{\GL}(K_6)&=-243 N^5 p_2+N^4 \left(376 p_1^2+361 p_2^2\right)+N^3 \left(-240 p_2^3-593 p_1^2 p_2-334 p_2+252 p_4\right)\\
                  &\phantom{=}+N^2 \left(160 p_1^4+405 p_2^2 p_1^2+232 p_1^2-1088 p_3 p_1+85 p_2^4+999 p_2^2-174 p_2 p_4\right)\\
                  &\phantom{=}+N (-15 p_2^5-130 p_1^2 p_2^3-537 p_2^3+30 p_4 p_2^2-165 p_1^4 p_2-159 p_1^2 p_2\\
                  &\phantom{=}+696 p_1 p_3 p_2-31 p_2+30 p_1^2 p_4+68 p_4-16 p_6)\\
                  &\phantom{=}+15 p_1^6+45 p_2^2 p_1^4+48 p_1^4-120 p_3 p_1^3+15 p_2^4 p_1^2+90 p_2^2 p_1^2\\
                  &\phantom{=}-120 p_2^2 p_3 p_1-192 p_3 p_1+96 p_5 p_1+p_2^6+90 p_2^4+144 p_2^2+160 p_3^2-240 p_2 p_4\\
    {w}_{\GL}(K_7)&=2022 N^6 p_2+N^5 \left(-3580 p_1^2-2947 p_2^2\right)+N^4 \left(1981 p_2^3+5446 p_1^2 p_2+5556 p_2-2808 p_4\right)\\
                  &\phantom{=}+N^3 \left(-1568 p_1^4-3773 p_2^2 p_1^2-5336 p_1^2+13280 p_3 p_1-770 p_2^4-14108 p_2^2+2408 p_2 p_4\right)\\
                  &\phantom{=}+N^2 (175 p_2^5+1435 p_1^2 p_2^3+8505 p_2^3-714 p_4 p_2^2+1750 p_1^4 p_2+5258 p_1^2 p_2\\
                  &\phantom{=}-10192 p_1 p_3 p_2+1862 p_2-644 p_1^2 p_4-2712 p_4+544 p_6)\\
                  &\phantom{=}+N (-210 p_1^6-735 p_2^2 p_1^4-1656 p_1^4+2576 p_3 p_1^3-280 p_2^4 p_1^2-2877 p_2^2 p_1^2\\
                  &\phantom{=}+210 p_2 p_4 p_1^2-524 p_1^2+2856 p_2^2 p_3 p_1+8800 p_3 p_1-3264 p_5 p_1-21 p_2^6-2177 p_2^4\\
                  &\phantom{=}-6985 p_2^2-4096 p_3^2+70 p_2^3 p_4+7292 p_2 p_4-112 p_2 p_6)\\
                  &\phantom{=}+p_2^7+21 p_1^2 p_2^5+210 p_2^5+105 p_1^4 p_2^3+630 p_1^2 p_2^3-280 p_1 p_3 p_2^3+2352 p_2^3-1680 p_4 p_2^2\\
                  &\phantom{=}+105 p_1^6 p_2+336 p_1^4 p_2+1120 p_3^2 p_2-840 p_1^3 p_3 p_2-4032 p_1 p_3 p_2+672 p_1 p_5 p_2+1344 p_1^2 p_4\\
  \end{align*}
\end{result}

%\newpage
\section{Hopf algebra structure and projection to primitives}\label{sec:ha}

Multiplicative weight systems often become simpler
when restricted to primitive elements in the Hopf algebra of chord diagrams.
This is true, in particular, for the weight systems associated to metrized Lie
algebras. The degree of the value of such a weight system $w_\fg$
on a chord diagram with~$n$ chords is~$2n$, while for the projection of the
chord diagram to primitives it is at most~$n$, see~\cite{ChV}.
In many cases, knowing the value of a weight system on projections to primitives
allows one to understand its structure.

In this section, we recall the Hopf algebra structure on the algebra of chord diagrams
modulo $4$-term relations, and discuss the values of~$w_\GL$ on projections of chord
diagrams to primitives.

\begin{definition}
The {\it coproduct} $\Delta$ of a chord diagram $D$ is defined by
\[
	\Delta(D):=\sum_{J\subseteq[D]} D_J\otimes D_{\bar{J}},
\]
where the summation is taken over all subsets $J$ of the set $[D]$ of chords of $D$. Here $D_J$ is the chord subdiagram
of~$D$ consisting of the chords that belong to $J$ and $\bar{J} = [D] \setminus J$ is the complementary subset of chords.
\end{definition}
\begin{claim} The algebra of chord diagrams modulo $4$-term relations
	endowed with the above coproduct
	is a graded commutative, cocommutative and connected Hopf algebra.
\end{claim}
\begin{definition}An element $p$ of a Hopf algebra is called {\it primitive} if $\Delta(p) = 1 \otimes p + p \otimes 1$.
\end{definition}

The Milnor--Moore theorem, when applied to the Hopf algebra of chord diagrams,
asserts that this Hopf algebra admits a decomposition into the direct sum of
the subspace of primitive elements and the subspace of decomposable elements
(polynomials in primitive elements of smaller degree).
There exists, therefore, a natural projection from the space of chord diagrams
to the subspace of primitive elements, whose kernel is the subspace of decomposable elements.
We denote this projection by~$\pi$.

\begin{theorem}[\cite{lando2000hopf,schmitt1994incidence}]
The projection $\pi(D)$ of a chord diagram $D$ to the subspace of primitive elements
 is given by the formula
\begin{align*}
	\pi(D) &= D-1!\sum_{[D_1]\sqcup [D_2]=[D]}D_1\cdot D_2 +2! \sum_{[D_1]\sqcup [D_2]\sqcup[D_3]=[D]}D_1\cdot D_2\cdot D_3\cdots\\
	&= D-\sum_{i=2}^{|[D]|}(-1)^i(i-1)!\sum_{\substack{\bigsqcup\limits_{j=1}^i [D_j]=[D]\\ [D_j]\ne \emptyset}}\prod_{j=1}^i D_j
\end{align*}
where the sum is taken over all unordered splittings of the set of chords
of $D$ into $2, 3,$ etc nonempty subsets.
\end{theorem}

In particular, the chord diagrams~$K_n$ generate a graded Hopf subalgebra in
the Hopf algebra of chord diagrams (since any chord subdiagram of~$K_n$ is $K_k$, for some~$k$).
Rewriting the formula for the projection for the exponential generating series $1+\sum_{n=1}^\infty{K_n}\frac{x^n}{n!}$
we obtain

\begin{corollary}
The generating series for the projections $\pi(K_n)$  to the subspace of primitive elements is given by the formula
\begin{equation}
	\sum_{n=1}\pi(K_n)\frac{x^n}{n!} = \log\left(1+\sum_{n=1}K_n\frac{x^n}{n!}\right)  \label{eq:projection}
\end{equation}
\end{corollary}

Now, knowing the values of the $\GL$ weight system on the diagrams~$K_n$ for
$n=1,2,\dots,7$, we easily obtain the values $\bar{w}_\GL=w_\GL\circ\pi$
on their projections to primitives:

\begin{result}
	\begin{align*}
		\bar{w}_{\GL}(K_2)&=-N C_2+C_1^2\\
		\bar{w}_{\GL}(K_3)&=2 N^2 C_2-2 N C_1^2\\
		\bar{w}_{\GL}(K_4)&=-6 C_2 N^3+(6 C_1^2-2 C_3) N^2+(6 C_1 C_2-2 C_2+2 C_4) N-4 C_1^3+2 C_1^2+6 C_2^2-8 C_1 C_3\\
		\bar{w}_{\GL}(K_5)&=24 C_2 N^4+(24 C_3-24 C_1^2) N^3+(-72 C_1 C_2+32 C_2-24 C_4) N^2\\
							&\phantom{=}+(48 C_1^3-32 C_1^2+96 C_3 C_1-72 C_2^2) N\\
		\bar{w}_{\GL}(K_6)&=-120 C_2 N^5+(120 C_1^2-240 C_3) N^4+(720 C_1 C_2-416 C_2+224 C_4) N^3\\
							&\phantom{=}+(-480 C_1^3+416 C_1^2-896 C_3 C_1+792 C_2^2-88 C_3+32 C_5) N^2\\
							&\phantom{=}+(-240 C_2 C_1^2+264 C_2 C_1-160 C_4 C_1-72 C_2+64 C_2 C_3+88 C_4-16 C_6) N\\
							&\phantom{=}+120 C_1^4-176 C_1^3+72 C_1^2-192 C_1 C_2^2+264 C_2^2+160 C_3^2+256 C_1^2 C_3\\
							&\phantom{=}-352 C_1 C_3-240 C_2 C_4+96 C_1 C_5\\	
		\bar{w}_{\GL}(K_7)&=720 C_2 N^6+(2400 C_3-720 C_1^2) N^5+(-7200 C_1 C_2+5264 C_2-1856 C_4) N^4\\
							&\phantom{=}+(4800 C_1^3-5264 C_1^2+7424 C_3 C_1-9168 C_2^2+3392 C_3-1088 C_5) N^3\\
							&\phantom{=}+(7200 C_2 C_1^2-10176 C_2 C_1+5440 C_4 C_1+3456 C_2-2176 C_2 C_3-3392 C_4+544 C_6) N^2\\
							&\phantom{=}+(-3600 C_1^4+6784 C_1^3-8704 C_3 C_1^2-3456 C_1^2+6528 C_2^2 C_1+13568 C_3 C_1\\
							&\phantom{=}-3264 C_5 C_1-10176 C_2^2-4096 C_3^2+6816 C_2 C_4) N+1344 C_2^3-2688 C_1 C_2 C_3+1344 C_1^2 C_4\\	
	\end{align*}
\end{result}

In the basis $p_1,p_2,\dots$ of shifted power series, these formulas look simpler:

\begin{result}
	\begin{align*}
		\bar{w}_{\GL}(K_2)&=-Np_2+p_1^2\\
		\bar{w}_{\GL}(K_3)&=2 N^2 p_2-2 N p_1^2\\
		\bar{w}_{\GL}(K_4)&=-7 N^3 p_2+8 N^2 p_1^2+N \left(2 p_4-p_2\right)+6 p_2^2-8 p_1 p_3\\
		\bar{w}_{\GL}(K_5)&=36 N^4 p_2-48 N^3 p_1^2+20 N^2 p_2+N \left(-8 p_1^2-72 p_2^2-24 p_4\right)+96 p_1 p_3\\
    \bar{w}_{\GL}(K_6)&=-243 N^5 p_2+376 N^4 p_1^2+N^3 \left(252 p_4-334 p_2\right)+N^2 \left(232 p_1^2-1088 p_3 p_1+864 p_2^2\right)\\
                      &\phantom{=}+N \left(-96 p_2 p_1^2-31 p_2+68 p_4-16 p_6\right)+48 p_1^4+144 p_2^2+160 p_3^2-192 p_1 p_3-240 p_2 p_4+96 p_1 p_5\\
    \bar{w}_{\GL}(K_7)&=2022 N^6 p_2-3580 N^5 p_1^2+N^4 \left(5556 p_2-2808 p_4\right)+N^3 \left(-5336 p_1^2+13280 p_3 p_1-11280 p_2^2\right)\\
                      &\phantom{=}+N^2 \left(2976 p_2 p_1^2+1862 p_2-2712 p_4+544 p_6\right)\\
                      &\phantom{=}+N \left(-1488 p_1^4-524 p_1^2+8800 p_3 p_1-3264 p_5 p_1-6768 p_2^2-4096 p_3^2+6816 p_2 p_4\right)\\
                      &\phantom{=}+1344 p_2^3-2688 p_1 p_2 p_3+1344 p_1^2 p_4\\
	\end{align*} 	
\end{result}


\begin{thebibliography}{10}

\bibitem{bar1995vassiliev}
Dror Bar-Natan,
\newblock \emph{On the Vassiliev knot invariants},
\newblock Topology, 34(2):423--472, 1995.
\newblock (an updated version available at
  \url{http://www.math.toronto.edu/~drorbn/papers}).

%\bibitem{bar1991}
%Dror Bar-Natan.
%\newblock Weights of Feynman diagrams and the Vassiliev knot invariants.
%\newblock preprint, February 1991.
%\newblock (an updated version available at
%  \url{http://www.math.toronto.edu/~drorbn/papers}).

\bibitem{CD}
S.~Chmutov, S.~Duzhin,
\newblock {\em A lower bound for the number of Vassiliev knot invariants.},
\newblock Topology and its Applications 92 (1999) 201–223.



\bibitem{Chmutov2012}
S.~Chmutov, S.~Duzhin, and J.~Mostovoy,
\newblock {\em Introduction to {Vassiliev Knot Invariants}},
\newblock Cambridge University Press, May 2012.

\bibitem{ChL}
 S.~Chmutov, S.~Lando,
\newblock \emph{Mutant knots and intersection graphs},
\newblock Algebr. Geom. Topol.
\newblock 2007, 7, 3, 1579--1598

\bibitem{ChV}
 S.~Chmutov, A.~Varchenko,
\newblock \emph{Remarks on the Vassiliev knot invariants coming from $\mathfrak{sl}_2$},
\newblock Topology
\newblock 1997, 36, 1, 153--178

\bibitem{Dasbach}
O.~Dasbach,
\newblock \emph{On the Combinatorial Structure of Primitive Vassiliev Invariants III — A Lower Bound},
\newblock Communications in Contemporary Mathematics, Vol. 2, No. 4, 2000, pp. 579–590.


\bibitem{FKV}
Figueroa-O'Farrill, T.~Kimura, A.~Vaintrob,
\newblock \emph{The Universal Vassiliev Invariant for the Lie Superalgebra $\GL(1|1)$},
\newblock Comm. Math. Phys.,
\newblock 1997, 185, 93--127

\bibitem{F1}
P.~Filippova,
\newblock \emph{Values of the $\SL_2$ Weight System on Complete Bipartite Graphs},
\newblock Funct. Anal. Appl.,
\newblock 54:3 (2020), 208--223

\bibitem{F2}
P.~Filippova,
\newblock  \emph{Values of the $\SL_2$-weight system on a family of graphs that are not intersection graphs
of chord diagrams},
\newblock Sbornik Math.
\newblock 213:2 (2022),  115--148

\bibitem{kon1993}
M.~Kontsevich,
\newblock \emph{Vassiliev knot invariants},
\newblock in: Advances in Soviet Math., 16(2):137--150, 1993.

\bibitem{lando2000hopf}
Sergei~K. Lando,
\newblock \emph{On a Hopf algebra in graph theory},
\newblock Journal of Combinatorial Theory, Series B, 80(1):104--121,
  2000.

\bibitem{Ols1996}
A.~Okounkov,  G.~Olshanski,
\newblock \emph{Shifted Schur Functions},
\newblock St. Petersburg Math. J. volume 9, no.~2, 239--300,
\newblock 1998.

\bibitem{Ols1991}
G.~Olshanski,
\newblock \emph{Representations of
infinite-dimensional classical groups,
limits of enveloping algebras and yangians},
\newblock in “Topics in Representation Theory”,
Advances in Soviet Math. 2,
\newblock Amer. Math. Soc., Providence RI, 1991, pp. 1--66.

%  \bibitem{Pen}
%R. Penrose
%\newblock  Applications of negative dimensional tensors.
%\newblock In: Combinatorial mathematics and its applications (ed. D. J. A. Welsh), New York.
%\newblock London: Academic Press, 1971.

\bibitem{PerelomovPopov1968}
A.~M.~Perelomov and Vladimir S.~Popov,
\newblock \emph{Casimir operators for semisimple groups},
\newblock  Math. USSR Izv. Volume 2(6), p.1313, 1968.


\bibitem{schmitt1994incidence}
William~R Schmitt,
\newblock \emph{Incidence Hopf algebras},
\newblock Journal of Pure and Applied Algebra, 96(3):299--330, 1994.

\bibitem{ZY}
Zhuoke Yang,
\newblock \emph{On values of $\SL_3$ weight system on chord diagrams whose intersection graph is complete bipartite},
\newblock arXiv:2102.00888


\bibitem{Za}
P.~Zakorko,
\newblock to appear


\bibitem{Zhe}
Zhelobenko D.P.
\newblock \emph{Compact Lie groups and their representations},
\newblock Nauka, Moscow, 1970.
\newblock English translation: Translations of mathematical monographs, v. 40.
\newblock American Mathematical Society, Providence, Rhode Island, 1973.



\end{thebibliography}
\end{document}